\documentclass[preprint,number,12pt]{elsarticle}
\usepackage{amsmath}
\usepackage{amssymb}
\usepackage{amsthm}
\usepackage{Factorial_Polynomials}
\begin{document}
\begin{frontmatter}
\title{Factorial Polynomials and Associated Number Families}
\author{Alfred Schreiber | \today}
\address{Department of Mathematics and Mathematical Education, University of Flensburg,\\Auf dem Campus 1, D-24943 Flensburg, Germany}
\ead{info@alfred-schreiber.de}
\begin{abstract}
Two doubly indexed families of polynomials in several indeterminates are considered. They are related to the falling and rising factorials in a similar way as the potential polynomials (introduced by L. Comtet) are related to the ordinary power function. We study the inversion relations valid for these factorial polynomials as well as the number families associated with them.
\end{abstract}
\begin{keyword}
Potential polynomials \sep Factorial polynomials \sep Inverse relations \sep Stirling numbers \sep Lah numbers
\enlargethispage{5ex}
\MSC[2010] 05A10 \sep 11B65 \sep 11B73 \sep 11B83 \sep 11C08 
\end{keyword}
\end{frontmatter}
%%
%% Main text:
%%-----------
%% Section 1: 
\section{Basic notions}\label{section1}
\noindent
In what follows, we denote by $\funcs$ the algebra $\const[[x]]$ of formal power series in $x$ with coefficients in a fixed commutative field $\const$ of characteristic zero. The elements of $\funcs$ will be called \emph{functions}, those of $\const$ \emph{constants} (or \emph{numbers}). For any functions $f,g$ and constant $c$, the sum $f+g$, the scalar product $cf$, and the product $f\cdot g$ are, as customary, assumed to be defined  coordinate-wise and by Cauchy convolution, respectively. Furthermore, we consider the composition $\comp$ with $(f\comp g)(x):=f(g(x))$ which is a partial operation on $\funcs$, but well-defined, for example, if the leading coefficient of $g$ is zero or, otherwise, $f$ is a (Laurent) polynomial in $x^{-1},x$ with coefficients in $\const$. The identity element is, of course, $\id=\id(x):=x$, satifying $f\comp\id=\id\comp f=f$.

The ordinary algebraic derivation $D$ on $\const[x]$ (with $D(\id)=1$) can be extended, in a unique way, to a derivation on $\funcs$ (here also denoted by $D$) for which the known rules for addition, multiplication and composition apply \cite[p.\,15, 61]{schr2021}. Iterating $D$ leads, in the usual way, to derivatives of higher order $D^n(f)$. By setting $f_n:=D^n(f)(0)$, we obtain, as is well known, the representation of $f(x)$ in the form of a Taylor series  expansion: 
\[
	f(x)=\sum_{n\geq0}f_n\frac{x^n}{n!}. 
\]
The constants $f_0,f_1,f_2,\ldots\in\const$ are called \emph{Taylor coefficients of} $f$.

\section{Potential polynomials}\label{section2}
\noindent
Of special interest is the task of finding the general Taylor coefficient of a composite function $f\comp g$.  Its well known solution is given by a famous formula of Fa\`{a} di Bruno (cf. \cite[p.\,137]{comt1974} and \cite[eqs. (1.3) and (4.1)]{schr2015}), which in modern notation is
\begin{equation}\label{faadibruno}
	D^n(f\comp g)(0)=\sum_{k=0}^{n}D^k(f)(g_0)B_{n,k}(g_1,\ldots,g_{n-k+1}).
\end{equation}
Here $B_{n,k}$ are the \emph{partial Bell polynomials} (or: \emph{exponential polynomials}) that can be represented in the form of a `diophantine' sum
\begin{equation*}
	B_{n,k}=\sum\frac{n!}{r_1!r_2!\dotsm(1!)^{r_1}(2!)^{r_2}\dotsm}X_1^{r_1}X_2^{r_2}\dotsm X_{n-k+1}^{r_{n-k+1}}
\end{equation*}
to be taken over all sequences of integers $r_1,r_2,r_3,\ldots\geq 0$ such that $r_1+r_2+r_3+\dotsb=k$ and $r_1+2r_2+3r_3+\dotsm=n$ (see \cite[p.\,134]{comt1974}, \cite[p.\,26, 65]{schr2021}). By replacing the Taylor coefficients $g_j$ with indeterminates $X_j$ on the right-hand side of \eqref{faadibruno}, one obtains a polynomial dependent on $f$ (in \cite{schr2020,schr2021} called the \emph{Fa\`{a} di Bruno polynomial of} $f$). An important special case is $f=\id^k$, $k\in\integers$, which leads to the \emph{potential polynomials} 
\begin{equation}\label{potentialpolynomials}
	P_{n,k}=\sum_{j=0}^{n}\ffac{k}{j}X_{0}^{k-j}B_{n,j}
\end{equation}
introduced by Comtet \cite[p.\,141]{comt1974} and extensively studied in \cite{comt1974,schr2020, schr2021}. Here $\ffac{k}{j}$ is D. Knuth's symbol for the falling factorial power $k(k-1)\cdots(k-j+1)$, and $\ffac{k}{j}=1$ for $j=0$. 

If $g\in\funcs$ is (compositionally) invertible, then $g_0=g(0)=0$, and we obtain from \eqref{faadibruno} $D^n(g^k)(0)=k!B_{n,k}(g_1,\ldots,g_{n-k+1})$, that is, $B_{n,k}(g_1,\ldots,g_{n-k+1})$ is the $n$th Taylor coefficient of $g(x)^{k}/k!$. For completeness, we note that there also exists a polynomial expression $A_{n,k}(g_1,\ldots,g_{n-k+1})$ for the $n$th Taylor coefficient of $\inv{g}(x)^k/k!$, where $\inv{g}$ denotes the unique inverse of $g$ (with $g\comp\inv{g}=\inv{g}\comp g=\id$). The fundamental properties (recurrences, inverse relations, reciprocity laws) of these two families of polynomials are treated in detail in \cite{schr2015,schr2020,schr2021}. Here we make use of the fact that the (lower triangular) matrices $(A_{n,k})$ and $(B_{n,k})$ are inverses of each other with respect to matrix multiplication, more precisely:
\begin{equation}\label{inversion_AB}
	\sum_{j=k}^{n}A_{n,j}B_{j,k}=\kronecker{n}{k}\quad(1\leq k\leq n),
\end{equation}
where $\kronecker{n}{n}=1$, $\kronecker{n}{k}=0$ if $n\neq k$ (Kronecker symbol); see \cite[p.\,29 and p.\,82]{schr2021}.

Let $Q_n$ be any sequence of polynomials from $\const[X_1,X_2,\ldots]$. We then call the sequence of numbers $q(n):=Q_n(1,1,\ldots)$, obtained by replacing each indeterminate occurring in $Q_n$ by 1, \emph{associated with} $Q_n$. Thus $q(n)$ equals the sum of the coefficients of $Q_n$. For example, it is well known that the family of numbers $s_2(n,k):=B_{n,k}(1,\ldots,1)$ associated with the partial Bell polynomials consists just of the Stirling numbers of the second kind \cite[Thm.\,B, p.\,135]{comt1974}. Together with the signed Stirling numbers of the first kind $s_1(n,k)$, they satisfy the inverse relation $\sum_{j=k}^{n}s_1(n,j)s_2(j,k)=\kronecker{n}{k}$ (see, for example, \cite[Prop.\,1.4.1(a)]{stan1986}). Therefore, specializing all indeterminates to 1 in \eqref{inversion_AB}, $A_{n,k}(1,\ldots,1)$ turns out to be $s_1(n,k)$.

In his famous treatise \textit{Methodus differentialis} (London, 1730) J. Stirling introduced the number family $s_1(n,k)$ by expanding $\ffac{x}{n}$ into a polynomial in standard form
\begin{equation}\label{stirlingnumbers_1}
	\ffac{x}{n}=x(x-1)\cdots(x-n+1)=\sum_{k=0}^{n}s_1(n,k)x^k.
\end{equation}
Due to the inverse relationship of the Stirling numbers of the first and second kind \cite[Prop.\,1.4.1(b)]{stan1986} we get from \eqref{stirlingnumbers_1}
\begin{equation}\label{stirlingnumbers_2}
	x^n=\sum_{k=0}^{n}s_2(n,k)\ffac{x}{k}.
\end{equation}
Comparing now \eqref{stirlingnumbers_2} with \eqref{potentialpolynomials}, we find that $P_{n,k}(1,\ldots,1)=k^n$ holds, i.e., $k^n$ is the number family associated with the potential polynomials $P_{n,k}$.

\begin{rem}
The power terms $k^n,\ffac{k}{n},\rfac{k}{n}$ have an obvious combinatorial meaning that entails a natural combinatorial interpretation of equations \eqref{stirlingnumbers_1} and \eqref{stirlingnumbers_2} (see, for example, \cite[p.\,33--35]{stan1986}). Also the Stirling numbers have a combinatorial meaning: $s_2(n,k)$ counts the ways $n$ objects can be divided into $k$ non-empty subsets (`subset numbers'), whereas the signless expression $c(n,k):=|s_1(n,k)|=(-1)^{n-k}s_1(n,k)$ represents the number of permutations of $n$ objects having $k$ cycles (`cycle numbers').
\end{rem}

\section{Factorial polynomials}\label{section3}
\noindent
Analogous to the way in which the potential polynomials were defined as Fa\`{a} di Bruno polynomials of $\id^k$, let us now introduce \emph{lower and upper factorial polynomials} as Fa\`{a} di Bruno polynomials of the falling and rising factorial power functions $\ffac{\id}{k}$ and $\rfac{\id}{k}$, respectively:
\begin{align}
		\label{facpoly_1}
		P_{n,\underline{k}}&:=\sum_{j=0}^{n}D^j(\ffac{\id}{k})(X_0)B_{n,j}(X_1,\ldots,X_{n-j+1})\\
		\label{facpoly_2}
		P_{n,\overline{k}}&:=\sum_{j=0}^{n}D^j(\rfac{\id}{k})(X_0)B_{n,j}(X_1,\ldots,X_{n-j+1})
\end{align}
We first consider \eqref{facpoly_1}. Since $D^j$ is a linear operator, applying \eqref{stirlingnumbers_2} to $\ffac{\id}{k}$ and observing \eqref{potentialpolynomials} yields
\begin{align}
	P_{n,\underline{k}}&=\sum_{j=0}^{n}\sum_{r=0}^{k}s_1(k,r)D^j(\id^r)(X_0)B_{n,j}\notag\\
			               &=\sum_{r=0}^{k}s_1(k,r)\sum_{j=0}^{n}\ffac{r}{j}X_{0}^{r-j}B_{n,j}\notag\\
										 \label{facpoly_1topotential}
										 &=\sum_{r=0}^{k}s_1(k,r)P_{n,r}\\
\intertext{which can be reversed to}
  \label{potentialtofacpoly_1}
	P_{n,k} &=\sum_{r=0}^{k}s_2(k,r)P_{n,\underline{r}}.
\end{align}
In a similar way \eqref{facpoly_2} can be evaluated. The only thing to keep in mind is that $\rfac{\id}{k}=(-1)^k\ffac{(-\id)}{k}$ and therefore by \eqref{stirlingnumbers_1}
\begin{align}
	\rfac{\id}{k}=(-1&)^k\sum_{r=0}^{k}s_1(k,r)(-\id)^r=\sum_{r=0}^{k}c(k,r)\id^r,\notag\\
\intertext{whence we readily obtain}
	\label{facpoly_2topotential}
							 P_{n,\overline{k}}=&\sum_{r=0}^{k}c(k,r)P_{n,r}.
\end{align}
To be able to establish also the connection between $P_{n,\overline{k}}$ and $P_{n,\underline{k}}$, recall the numbers $l(n,k)$ introduced by I.\,Lah \cite{lah1955}, \cite[p.\,43--44]{rior2002}: 
\begin{equation}\label{lahnumbers}
	l(n,k):=(-1)^n\frac{n!}{k!}\binom{n-1}{k-1}\quad\text{with}\quad\rfac{x}{n}=\sum_{k=0}^{n}l(n,k)\ffac{x}{k}.
\end{equation}
\begin{rem}
	The unsigned \emph{Lah numbers} $|l(n,k)|$ count the ways a set of $n$ objects can be partitioned into $k$ non-empty linearly ordered subsets.
\end{rem}
Two well-known remarkable properties of the signed Lah numbers come into play in what follows: their representability by the Stirling numbers, and the fact that they are inverses of themselves (see \cite[p.\,44]{rior2002} and \cite[p.\,31,\,91]{schr2021}):
\begin{align}
	\label{lahbystirling}
	\sum_{j=k}^n (-1)^j s_1(n,j)s_2(j,k)&=l(n,k),\\
	\label{lahselfinverse}
	\sum_{j=k}^{n}l(n,j)l(j,k)&=\kronecker{n}{k}.
\end{align}
\begin{prop}\label{facpoly_1_and_2}
For all non-negative integers $n,k$ we have
\begin{align*}
\text{\emph{(i)}~~~}P_{n,\overline{k}}&=\sum_{j=0}^k(-1)^{k}l(k,j)P_{n,\underline{j}},\\
\text{\emph{(ii)}~~~}P_{n,\underline{k}}&=\sum_{j=0}^k(-1)^{j}l(k,j)P_{n,\overline{j}}.
\end{align*}
\end{prop}
\begin{proof}
We have
\begin{align*}
	P_{n,\overline{k}}&=\sum_{r=0}^{k}c(k,r)P_{n,r}&\text{(by \eqref{facpoly_2topotential})}\\
										&=\sum_{r=0}^{k}\sum_{j=0}^{r}c(k,r)s_2(r,j)P_{n,\underline{j}}&\text{(by \eqref{potentialtofacpoly_1})}\\
										&=\sum_{r=0}^{k}\sum_{j=0}^{r}(-1)^{k-r}s_1(k,r)s_2(r,j)P_{n,\underline{j}}&\text{(by Remark 2.1)}\\
										&=\sum_{j=0}^{k}(-1)^k\left(\sum_{r=j}^{k}(-1)^{r}s_1(k,r)s_2(r,j)\right)P_{n,\underline{j}}&\text{(rearranging the double sum)}.
\end{align*}
Now, applying \eqref{lahbystirling} to the inner sum, we obtain assertion (i).\\
Assertion (ii) follows from (i) with respect to \eqref{lahselfinverse} as follows:
\begin{align*}
 \sum_{j=0}^{k}(-1)^{j}l(k,j)P_{n,\overline{j}}&=\sum_{j=0}^{k}(-1)^{j}l(k,j)\sum_{r=0}^{j}(-1)^{j}l(j,r)P_{n,\underline{r}}\\
	&=\sum_{r=0}^{k}\left(\sum_{j=r}^{k}(-1)^{2j}l(k,j)l(j,r)\right)P_{n,\underline{r}}\\
	&=\sum_{r=0}^{k}\kronecker{k}{r}P_{n,\underline{r}}=P_{n,\underline{k}}.\qedhere
\end{align*}
\end{proof}
%%

%% Section 4: 
\section{Associated number families}\label{section4}
\noindent
This section deals with the number families associated with the upper and lower factorial polynomials:
\begin{align}
	\label{facnum_upper}
	[n]_{\overline{k}}&:=P_{n,\overline{k}}(1,\ldots,1)=\sum_{r=0}^{k}c(k,r)r^n,\\
	\label{facnum_lower}
	[n]_{\underline{k}}&:=P_{n,\underline{k}}(1,\ldots,1)=\sum_{r=0}^{k}s_1(k,r)r^n.
\end{align}
To be more consistent with the combinatorial interpretation of the Stirling numbers (as explained in Remark 2.1), we will use in some cases the notation of J. Karamata recommended by Knuth \cite{knut1992}:
\begin{equation*}
	s_2(n,k)=\subsets{n}{k} \qquad \text{and} \qquad c(n,k)=\cycles{n}{k}.
\end{equation*}
We start with some preparatory remarks. From \eqref{stirlingnumbers_2} we readily obtain $r^n=\sum_{j=0}^{n}j!\binom{r}{j}s_2(n,j)$ which implies an identity that reduces a power sum with arbitrary weights $\gamma_{k,r}$, $1\leq r\leq k$, and integer exponent $n\geq 1$ as follows:
\begin{equation}\label{powersum}
	\sum_{r=1}^{k}\gamma_{k,r}r^n = \sum_{j=1}^{\min(k,n)}j!s_2(n,j)\sum_{r=j}^{k}\binom{r}{j}\gamma_{k,r}.
\end{equation}
There are quite a few cases allowing the inner sum on the right-hand side of \eqref{powersum} to be simplified significantly, that is, we would obtain an upper summation rule such as
\begin{equation}\label{sumrule}
	\sum_{r=j}^{k}\binom{r}{j}\gamma_{k,r}=f(k,j)
\end{equation}
with a more or less closed expression $f(k,j)$. 

\begin{rem}
The most simple example is $\gamma_{k,r}\equiv 1$, yielding in \eqref{sumrule} the term $f(k,j)=\binom{k+1}{j+1}$. With this, \eqref{powersum} becomes the familar identity
\begin{equation}\label{psw1}
	1^n+2^n+\cdots+k^n = \sum_{j=1}^{\min(k,n)}j!\subsets{n}{j}\binom{k+1}{j+1}.
\end{equation}
The reader may find in Hsu \cite{hsu1991} a great variety of choices for the $\gamma_{k,r}$. However, the case of Stirling numbers as weights has not been considered in those discussions.
\end{rem}

Boyadzhiev \cite{boya2009} has evaluated \eqref{powersum} for the cycle numbers $\gamma_{k,r}=c(k,r)$ thereby taking $f(k,j)=c(k+1,j+1)$, which turns \eqref{sumrule} into a well-known identity (see, for example, \cite[p.\,68, eq.\,(51)]{knut1997} ). This gives immediately the following nice result \cite[Prop.\,2.7]{boya2009} that has a remarkable analogy to \eqref{psw1}:
\begin{prop}\label{facnum_upper_explicit}
\[
	[n]_{\overline{k}}=\sum_{r=1}^{k}\cycles{k}{r}r^n=\sum_{j=1}^{\min(k,n)}j!\subsets{n}{j}\cycles{k+1}{j+1}.
\]	
\end{prop}
\begin{rem}
Instead of $s_2(n,j)$ Boyadzhiev makes use of a Stirling \emph{function} $S(n,j)$ whose first argument is allowed to be a complex number $n\neq 0$. In this case, the upper summation limit $\min(k,n)$, which appears in \eqref{powersum} and in Proposition \ref{facnum_upper_explicit}, has to be replaced by $k$.
\end{rem}
From now on, the exponent $n$ is assumed to be any positive integer.

Let us turn finally to the question of what result we get from \eqref{powersum} when the \emph{signed} Stirling numbers of the first kind $s_1(k,r)$ are chosen as weights. The answer is given in the following
\begin{prop}\label{facnum_lower_explicit}
\[
	[n]_{\underline{k}}=\sum_{r=1}^{k}(-1)^{k-r}\cycles{k}{r}r^n=\sum_{j=1}^{\min(k,n)}(-1)^{k-j}j!\subsets{n}{j}\left(\cycles{k-1}{j-1}-\cycles{k-1}{j}\right).
\]	
\end{prop}
\begin{proof}
Evaluating \eqref{powersum} for $\gamma_{k,r}=s_1(k,r)$ requires a new upper summation formula:
\begin{equation}\label{prf1}
	\sum_{r=j+1}^{k}\binom{r}{j}s_1(k,r)=k s_1(k-1,j).
\end{equation}
First, we show by induction on $k$ that \eqref{prf1} holds for every $k\geq 1$ and $0\leq j\leq k-1$. In the case of $k=1$ we have $j=0$, and both sides of \eqref{prf1} are equal to 1. 

The induction step ($k\to k+1$) makes repeated use of the familiar recurrence formula for the Stirling numbers of the first kind (see, e.\,g., \cite[p.\,67]{knut1997}):\\
$s_1(k+1,r)=s_1(k,r-1)-k s_1(k,r)$. Replacing $s_1(k+1,r)$ by $s_1(k,r-1)-k s_1(k,r)$ we obtain
\begin{equation}\label{prf2}
	\sum_{r=j+1}^{k+1}s_1(k+1,r)\binom{r}{j}=\sum_{r=j+1}^{k+1}s_1(k,r-1)\binom{r}{j}-\sum_{r=j+1}^{k+1}k s_1(k,r)\binom{r}{j}.
\end{equation}
By induction hypothesis, the second sum on the right-hand side of \eqref{prf2} is equal to $k^{2}s_1(k-1,j)$. The first sum can be splitted into two parts by applying the familiar basic addition formula for the binomial coefficients. We then have, again by induction hypothesis,
\begin{align*}
	\sum_{r=j+1}^{k+1}s_1(k,r-1)\binom{r-1}{j-1}&=\sum_{r=j}^{k}s_1(k,r)\binom{r}{j-1}=k s_1(k-1,j-1),\\
	\sum_{r=j+1}^{k+1}s_1(k,r-1)\binom{r-1}{j}&=\sum_{r=j}^{k}s_1(k,r)\binom{r}{j}=s_1(k,j)+k s_1(k-1,j).
\end{align*}
Combining these results we obtain for the sum on the left-hand side of \eqref{prf2}
\begin{align*}
	k s_1(k-1,j-1)+s_1(k,j)+k s_1(k-1,j)-k^2 s_1(k-1,j) &=\\
	k\cdot\left[s_1(k-1,j-1)-(k-1)s_1(k-1,j)\right]+s_1(k,j)&=\\
	k s_1(k,j) + s_1(k,j)&=\\
	(k+1)s_1(k,j),
\end{align*}
which completes the induction.

Finally, it follows from \eqref{prf1}: 
\begin{align*}
	\sum_{r=j}^{k}\binom{r}{j}s_1(k,r)&=s_1(k,j)+k s_1(k-1,j)\\
																		&=s_1(k-1,j-1)+s_1(k-1,j)\\
																	  &=(-1)^{k-j}\cycles{k-1}{j-1}-(-1)^{k-j}\cycles{k-1}{j}.
\end{align*}
This yields the asserted equation.
\end{proof}
\begin{rem}
For an alternative proof of Proposition \ref{facnum_lower_explicit} cf. \cite[p.\,253--254]{schr2014}.
\end{rem}
\begin{rem}
One finds a table of $P_{4,\underline{k}}$ for $1\leq k\leq 4$ as well as a table of $[n]_{\underline{k}}$ for $1\leq n,k\leq 7$ in \cite[p.\,131--132]{schr2021}.
\end{rem}
We conclude with some examples for small exponents $n=1,2,3$ which demonstrate that the given power sum indeed undergoes a substantial simplification through the right-hand side expression in the statement of Proposition \ref{facnum_lower_explicit}. 

Let $k\geq 2$, then, cancelling $(-1)^k$ in each equation we obtain the identities:
\begin{alignat*}{3}
	&\sum_{r=1}^{k}(-1)^{r}\cycles{k}{r}r   &&= \cycles{k-1}{1},\\
	&\sum_{r=1}^{k}(-1)^{r}\cycles{k}{r}r^2 &&= 3\cycles{k-1}{1} - 2\cycles{k-1}{2},\\
  &\sum_{r=1}^{k}(-1)^{r}\cycles{k}{r}r^3 &&= 7\cycles{k-1}{1} - 12\cycles{k-1}{2} + 6\cycles{k-1}{3}.
\end{alignat*}

%%
%% ----------------------------------------------------
\appendix
%%-------------------------
%% Bibliographic references
%% Last update 2023-06-24
%%-------------------------

%%-----------------------
\nocite* %% causes bibtex to take over all bibliographic entries
\end{document}